\documentclass[11pt,twoside]{article}
\usepackage{amsmath, amsthm, amscd, amsfonts, amssymb, graphicx, color}

\date{}

\begin{document}

\centerline{}

\centerline {\Large{\bf Some properties of $K$-frame in $n$-Hilbert space}}

\newcommand{\mvec}[1]{\mbox{\bfseries\itshape #1}}
\centerline{}
\centerline{\textbf{Prasenjit Ghosh}}
\centerline{Department of Pure Mathematics, University of Calcutta,}
\centerline{35, Ballygunge Circular Road, Kolkata, 700019, West Bengal, India}
\centerline{e-mail: prasenjitpuremath@gmail.com}
\centerline{}
\centerline{\textbf{T. K. Samanta}}
\centerline{Department of Mathematics, Uluberia College,}
\centerline{Uluberia, Howrah, 711315,  West Bengal, India}
\centerline{e-mail: mumpu$_{-}$tapas5@yahoo.co.in}

\newtheorem{Theorem}{\quad Theorem}[section]

\newtheorem{definition}[Theorem]{\quad Definition}

\newtheorem{theorem}[Theorem]{\quad Theorem}

\newtheorem{remark}[Theorem]{\quad Remark}

\newtheorem{corollary}[Theorem]{\quad Corollary}

\newtheorem{note}[Theorem]{\quad Note}

\newtheorem{lemma}[Theorem]{\quad Lemma}

\newtheorem{example}[Theorem]{\quad Example}

\newtheorem{result}[Theorem]{\quad Result}
\newtheorem{conclusion}[Theorem]{\quad Conclusion}

\newtheorem{proposition}[Theorem]{\quad Proposition}

\begin{abstract}
\textbf{\emph{The notion of a K-frame in n-Hilbert space is presented and some of their characterizations are given.\,We verify that sum of two K-frames is also a K-frame in n-Hilbert space.\,Also, the concept of tight K-frame in n-Hilbert space is described and some properties of its are going to be established.}}
\end{abstract}
{\bf Keywords:}  \emph{Frame, $K$-frame, n-normed space, n-inner product space, Frame  \\ \smallskip \hspace{2cm} in n-inner product space.}

{\bf 2010 Mathematics Subject Classification:} \emph{42C15, 46C07, 46C50.}
\\

\section{Introduction}
 
In 1952, Duffin and Schaeffer introduced frames in Hilbert spaces in their fundamental paper \cite{Duffin}, they used frames as a tool in the study of nonharmonic Fourier series.\,Later in 1986, the formal definition of frame in the abstract Hilbert spaces were given by Daubechies, Grossman, Meyer \cite{Daubechies}.\,A frame for a Hilbert space is a generalization of an orthonormal basis and this is such a tool that also allows each vector in this space can be written as a linear combination of elements from the frame but, linear independence among the frame elements is not required.\;Such frames play an important role in Gabor and wavelet analysis.\;Several generalizations of frames  namely, \,$g$-frame \cite{Sun}, \,$K$-frames \cite{Gavruta} etc. have been introduced in recent times.\;$K$-frames for a separable Hilbert spaces were introduced by Lara Gavruta to study the basic notions about atomic system for a bounded linear operator.\;In recent times, \,$K$-frame was presented to reconstruct elements from the range of a bounded linear operator\,$K$\, in a separable Hilbert space.\;$K$-frames are more generalization than the ordinary frames and many properties of ordinary frames may not holds for such generalization of frames.\,Generalized atomic subspaces for operators in Hilbert spaces were studied by P.\,Ghosh and T.\,K.\,Samanta {\cite{Ghosh}} and they were also presented the stability of dual \,$g$-fusion frames in Hilbert spaces in {\cite{P}}. 

The concept of \,$2$-inner product space was first introduced by Diminnie, Gahler and White \cite{Diminnie} in 1970's.\;In 1989, A.\,Misiak \cite{Misiak} developed the generalization of a \,$2$-inner product space for \,$n \,\geq\, 2$.\;Frame in \,$n$-Hilbert space was presented by P. Ghosh and T. K. Samanta \cite{Prasenjit}.\,They also discussed frame in tensor product of \,$n$-Hilbert spaces \cite{G}.

In this paper, we shall present the notion of a \,$K$-frame relative to \,$n$-Hilbert space and discuss some properties.\;Further it will be seen that the family of all \,$K$-frames is closed with respect to addition in \,$n$-Hilbert space.\;We also give the notion of a tight \,$K$-frame in \,$n$-Hilbert space.

Throughout this paper,\;$H$\; will denote a separable Hilbert space with the inner product \,$\left <\,\cdot \,,\, \cdot\,\right>$\, and \,$\mathcal{B}\,(\,H\,)$\; denote the space of all bounded linear operator on \,$H$.\;We also denote \,$\mathcal{R}\,(\,T\,)$\; for range set of \,$T$\; where \,$T \,\in\, \mathcal{B}\,(\,H\,)$\, and \;$l^{\,2}\,(\,\mathbb{N}\,)$\; denote the space of square summable scalar-valued sequences with  index set \,$\mathbb{N}$.

\section{Preliminaries}

\begin{theorem}(\,Douglas' factorization theorem\,)\,{\cite{Douglas}}\label{th1}
Let \,$U,\, V \,\in\, \mathcal{B}\,(\,H\,)$.\;Then the following conditions are equivalent:
\begin{itemize}
\item[(\,I\,)]\;\;$\mathcal{R}\,(\,U\,) \,\subseteq\, \mathcal{R}\,(\,V\,)$.
\item[(\,II\,)]\;\;$U\,U^{\,\ast} \,\leq\, \lambda^{\,2}\; V\,V^{\,\ast}$\, for some \,$\lambda \,>\, 0$.
\item[(\,III\,)]\;\;$U \,=\, V\,W$\, for some \,$W \,\in\, \mathcal{B}\,(\,H\,)$.
\end{itemize}
\end{theorem}

\begin{theorem}{\cite{Williams}}\label{thm1}
Let \,$S,\, T,\, U \,\in\, \mathcal{B}\,(\,H\,)$.\;Then the following are equivalent:
\begin{itemize}
\item[(\,I\,)]\;\;$\mathcal{R}\,(\,S\,) \,\subseteq\, \mathcal{R}\,(\,T\,) \,+\, \mathcal{R}\,(\,U\,)$.
\item[(\,II\,)]\;\;$S\,S^{\,\ast} \,\leq\, \lambda^{\,2}\;(\, T\,T^{\,\ast} \,+\, U\,U^{\,\ast}\,)$\, for some \,$\lambda \,>\, 0$.
\item[(\,III\,)]\;\;$S \,=\, T\,A \,+\, U\,B$\, for some \,$A \,,\, B \,\in\, \mathcal{B}\,(\,H\,)$.
\end{itemize}
\end{theorem}

\begin{theorem}(\cite{Christensen})\label{thm1.1}
Let \,$H_{\,1},\, H_{\,2}$\; be two Hilbert spaces and \;$U \,:\, H_{\,1} \,\to\, H_{\,2}$\; be a bounded linear operator with closed range \;$\mathcal{R}_{\,U}$.\;Then there exists a bounded linear operator \,$U^{\dagger} \,:\, H_{\,2} \,\to\, H_{\,1}$\, such that \,$U\,U^{\dagger}\,x \,=\, x\; \;\forall\; x \,\in\, \mathcal{R}_{\,U}$.
\end{theorem}

\begin{note}
The operator \,$U^{\dagger}$\, defined in Theorem (\ref{thm1.1}), is called the pseudo-inverse of \,$U$.
\end{note}

\begin{definition}\cite{Christensen}
A sequence \,$\left\{\,f_{\,i}\,\right\}_{i \,=\, 1}^{\infty} \,\subseteq\, H$\, is said to be a frame for \,$H$\, if there exist constants \,$0 \,<\, A \,\leq\, B \,<\, \infty$\, such that
\[ A\; \|\,f\,\|^{\,2} \,\leq\, \;\sum\limits_{i \,=\, 1}^{\infty}\;  \left|\ \left <\,f \,,\, f_{\,i} \, \right >\,\right|^{\,2} \,\leq\, B \;\|\,f\,\|^{\,2}\; \;\forall\; f \,\in\, H \]
The constants \,$A$\, and \,$B$\, are called frame bounds.\;If \,$\left\{\,f_{\,i}\,\right\}_{i \,=\, 1}^{\infty}$\; satisfies the inequality 
\[\sum\limits_{i \,=\, 1}^{\infty}\;  \left|\,\left <\,f \,,\, f_{\,i} \, \right >\,\right|^{\,2} \,\leq\, B \;\|\,f\,\|^{\,2}\; \;\forall\; f \,\in\, H\]
then it is called a Bessel sequence with bound \,$B$.
\end{definition}

\begin{definition}\cite{Christensen}
Let \,$\left\{\,f_{\,i}\,\right\}_{i \,=\, 1}^{\infty}$\; be a frame for \,$H$.\;Then the bounded linear operator \,$ T \,:\, l^{\,2}\,(\,\mathbb{N}\,) \,\to\, H$, defined by \,$T\,\{\,c_{i}\,\} \,=\, \sum\limits_{i \,=\, 1}^{\infty} \;c_{\,i}\,f_{\,i}$, is called  pre-frame operator and its adjoint operator \,$T\,^{\,\ast} \,:\, H \,\to\, l^{\,2}\,(\,\mathbb{N}\,)$, given by \;$T^{\,\ast}\,(\,f\,) \,=\, \left \{\, \left <\,f \,,\, f_{i} \,\right >\,\right \}_{i \,=\, 1}^{\infty}$\; is called the analysis operator.\;The operator \,$S \,:\, H \,\to\, H$\; defined by \,$S\,(\,f\,) \,=\, T\,T^{\,\ast}\,(\,f\,) \,=\, \sum\limits^{\infty}_{i \,=\, 1}\; \left <\,f \,,\, f_{\,i} \, \right >\,f_{\,i}\; \,\forall\; f \,\in\, H$\; is called the frame operator.
\end{definition}

\begin{definition}{\cite{Gavruta}}
Let \,$K \,\in\, \mathcal{B}\,(\,H\,)$.\;Then a sequence \,$\{\,f_{\,i}\,\}_{i \,=\, 1}^{\infty}$\, in \,$H$\, is said to be a \,$K$-frame for \,$H$\, if there exist constants \,$0 \,<\, A \,\leq\, B \,<\, \infty$\, such that
\[A \,\left \|\,K^{\,\ast}\,f\, \right \|^{\,2} \,\leq\, \sum\limits^{\infty}_{i \,=\, 1}\, \left |\,\left <\,f \,,\, f_{\,i}\,\right >\,\right|^{\,2} \,\leq\, B\;\left\|\,f\,\right\|^{\,2}\; \;\forall\; f \,\in\, H.\]
\end{definition}

\begin{theorem}{\cite{Xiao}}\label{th1.01}
Let \,$K$\, be a bounded linear operator on \,$H$.\;Then a Bessel sequence \,$\{\,f_{\,i}\,\}_{i \,=\, 1}^{\infty}$\, in \,$H$\; is a \,$K$-frame if and only if there exists \,$\lambda \,>\, 0$\, such that \,$S \,\geq\, \lambda\; K\,K^{\ast}$, where \,$S$\, is the frame operator for \,$\{\,f_{\,i}\,\}_{i \,=\, 1}^{\infty}$.
\end{theorem}

\begin{definition}\cite{Mashadi}
A real valued function \,$\left\|\,\cdot \,,\, \cdots \,,\, \cdot \,\right\| \,:\, X^{\,n} \,\to\, \mathbb{R}$\; is called a n-norm on \,$X$\; if the following conditions hold:
\begin{itemize}
\item[(I)]\;\; $\left\|\,x_{\,1} \,,\, x_{\,2} \,,\, \cdots \,,\, x_{\,n}\,\right\| \,=\, 0$\; if and only if \,$x_{\,1},\, \cdots,\, x_{\,n}$\, are linearly dependent,
\item[(II)]\;\;\; $\left\|\,x_{\,1} \,,\, x_{\,2} \,,\, \cdots \,,\, x_{\,n}\,\right\|$\; is invariant under any permutations of \,$x_{\,1},\, x_{\,2},\, \cdots,\, x_{\,n}$,
\item[(III)]\;\;\; $\left\|\,\alpha\,x_{\,1} \,,\, x_{\,2} \,,\, \cdots \,,\, x_{\,n}\,\right\| \,=\, |\,\alpha\,|\, \left\|\,x_{\,1} \,,\, x_{\,2} \,,\, \cdots \,,\, x_{\,n}\,\right\|\;\; \;\forall \;\; \alpha \,\in\, \mathbb{K}$,
\item[(IV)]\;\; $\left\|\,x \,+\, y \,,\, x_{\,2} \,,\, \cdots \,,\, x_{\,n}\,\right\| \,\leq\, \left\|\,x \,,\, x_{\,2} \,,\, \cdots \,,\, x_{\,n}\,\right\| \,+\,  \left\|\,y \,,\, x_{\,2} \,,\, \cdots \,,\, x_{\,n}\,\right\|$.
\end{itemize}
The pair \,$\left(\,X \,,\, \left\|\,\cdot \,,\, \cdots \,,\, \cdot \,\right\|\,\right)$\; is then called a linear n-normed space. 
\end{definition}

\begin{definition}\cite{Misiak}
Let \,$n \,\in\, \mathbb{N}$\; and \,$X$\, be a linear space of dimension greater than or equal to \,$n$\; over the field \,$\mathbb{K}$, where \,$\mathbb{K}$\, is the real or complex numbers field.\;A function \,$\left<\,\cdot \,,\, \cdot \,|\, \cdot \,,\, \cdots \,,\, \cdot\,\right> \,:\, X^{\,n \,+\, 1} \,\to\,  \mathbb{K}$\; is satisfying the following five properties:
\begin{itemize}
\item[(I)]\;\; $\left<\,x_{\,1} \,,\, x_{\,1} \;|\; x_{\,2} \,,\, \cdots \,,\, x_{\,n} \,\right> \;\geq\;  0$\; and \;$\left<\,x_{\,1} \,,\, x_{\,1} \;|\; x_{\,2} \,,\, \cdots \,,\, x_{\,n} \,\right> \;=\;  0$\; if and only if \;$x_{\,1},\, x_{\,2},\, \cdots,\, x_{\,n}$\; are linearly dependent,
\item[(II)]\;\; $\left<\,x \,,\, y \;|\; x_{\,2} \,,\, \cdots \,,\, x_{\,n} \,\right> \;=\; \left<\,x \,,\, y \;|\; x_{\,i_{\,2}} \,,\, \cdots \,,\, x_{\,i_{\,n}} \,\right> $\, for every permutation \;$\left(\, i_{\,2},\, \cdots,\, i_{\,n} \,\right)$\; of \;$\left(\,2,\, \cdots,\, n \,\right)$,
\item[(III)]\;\; $\left<\,x \,,\, y \;|\; x_{\,2} \,,\, \cdots \,,\, x_{\,n} \,\right> \;=\; \overline{\left<\,y \,,\, x \;|\; x_{\,2} \,,\, \cdots \,,\, x_{\,n} \,\right> }$,
\item[(IV)]\;\; $\left<\,\alpha\,x \,,\, y \;|\; x_{\,2} \,,\, \cdots \,,\, x_{\,n} \,\right> \;=\; \alpha \,\left<\,x \,,\, y \;|\; x_{\,2} \,,\, \cdots \,,\, x_{\,n} \,\right> $, \;for all \;$ \alpha \;\in\; \mathbb{K}$,
\item[(V)]\;\; $\left<\,x \,+\, y \,,\, z \;|\; x_{\,2} \,,\, \cdots \,,\, x_{\,n} \,\right> \;=\; \left<\,x \,,\, z \;|\; x_{\,2} \,,\, \cdots \,,\, x_{\,n} \,\right> \,+\,  \left<\,y \,,\, z \;|\; x_{\,2} \,,\, \cdots \,,\, x_{\,n} \,\right>$.
\end{itemize}
is called an n-inner product on \,$X$\, and the pair \,$\left(\,X \,,\, \left<\,\cdot \,,\, \cdot \,|\, \cdot \,,\, \cdots \,,\, \cdot\,\right> \,\right)$\; is called n-inner product space.
\end{definition}

\begin{theorem}\cite{Gunawan}
For \,$n$-inner product space \,$\left(\,X \,,\, \left<\,\cdot \,,\, \cdot \,|\, \cdot \,,\, \cdots \,,\, \cdot\,\right>\,\right)$, 
\[\left|\,\left<\,x \,,\, y \,|\, x_{\,2} \,,\,  \cdots \,,\, x_{\,n}\,\right>\,\right| \,\leq\, \left\|\,x \,,\, x_{\,2} \,,\, \cdots \,,\, x_{\,n}\,\right\|\, \left\|\,y \,,\, x_{\,2} \,,\, \cdots \,,\, x_{\,n}\,\right\|\]
hold for all \,$x,\, y,\, x_{\,2},\, \cdots,\, x_{\,n} \,\in\, X$.
\end{theorem}

\begin{theorem}\cite{Misiak}
For every n-inner product space \,$\left(\,X \,,\, \left<\,\cdot \,,\, \cdot \;|\; \cdot \,,\, \cdots \,,\, \cdot\,\right> \,\right)$,
\[\left \|\,x_{\,1} \,,\, x_{\,2} \,,\, \cdots \,,\, x_{\,n}\,\right\| \,=\, \sqrt{\left <\,x_{\,1} \,,\, x_{\,1} \;|\; x_{\,2} \,,\,  \cdots \,,\, x_{\,n}\,\right>}\] defines a n-norm for which
\[ \left <\,x \,,\, y \,|\, x_{\,2} \,,\,  \cdots \,,\, x_{\,n}\,\right> \,=\,\dfrac{\,1}{\,4}\, \left(\,\|\,x \,+\, y \,,\, x_{\,2} \,,\, \cdots \,,\, x_{\,n}\,\|^{\,2} \,-\, \|\,x \,-\, y \,,\, x_{\,2} \,,\, \cdots \,,\, x_{\,n}\,\|^{\,2}\,\right), \;\&\] 
\[\|\,x \,+\, y \,,\, x_{\,2} \,,\, \cdots \,,\, x_{\,n}\,\|^{\,2} \,+\, \|\,x \,-\, y \,,\, x_{\,2} \,,\, \cdots \,,\, x_{\,n}\,\|^{\,2}\]
\[ \,=\, 2\, \left(\,\|\,x \,,\, x_{\,2} \,,\, \cdots \,,\, x_{\,n}\,\|^{\,2} \,+\, \|\,y \,,\, x_{\,2} \,,\, \cdots \,,\, x_{\,n}\,\|^{\,2} \,\right)\] 
hold for all \,$x,\, y,\, x_{\,1},\, x_{\,2},\, \cdots,\, x_{\,n} \,\in\, X$.
\end{theorem}

\begin{definition}\cite{Aleksander}
Let \,$\left(\,X \,,\, \left<\,\cdot \,,\, \cdot \;|\; \cdot \,,\, \cdots \,,\, \cdot\,\right> \,\right)$\, be a \,$n$-inner product space and \,$\{\,e_{\,i}\,\}^{\,n}_{\,i \,=\, 1}$\; be linearly independent vectors in \,$X$.\;Then for a given set \,$F \,=\, \left\{\,a_{\,2} \,,\, \cdots \,,\, a_{\,n}\,\right\} \,\subseteq\, X$, if 
\[\left<\,e_{\,i} \,,\, e_{\,j} \,|\, a_{\,2} \,,\, \cdots \,,\, a_{\,n}\,\right> \,=\, \delta_{\,i\,j} \hspace{0.5cm} i,\; j \;\in\; \{\,1,\, 2,\, \cdots,\,n\,\}\]
\[\text{where}, \hspace{1.0cm} \delta_{\,i\,j} \;\;=\;\; \begin{cases}
1 & \text{if\;\;}\; i \;=\; j \\ 0 & \text{if\;\;}\; i \;\neq\; j \end{cases} \;, \hspace{4.5cm}\]
the family \;$\{\,e_{\,i}\,\}^{\,n}_{\,i \,=\, 1}$\; is said to be \,$F$-orthonormal.\;If an \,$F$-orthonormal set is countable, we can arrange it in the form of a sequence \;$\{\,e_{\,i}\,\}$\; and call it \,$F$-orthonormal sequence.
\end{definition}

\begin{definition}\cite{Mashadi}
A sequence \,$\{\,x_{\,k}\,\}$\; in a linear\;$n$-normed space \,$X$\, is said to be convergent to some \,$x \,\in\, X$\; if for every \,$ x_{\,2},\, \cdots,\, x_{\,n} \,\in\, X$\,  
\[\lim\limits_{k \to \infty}\,\left\|\,x_{\,k} \,-\, x \,,\, x_{\,2} \,,\, \cdots \,,\, x_{\,n} \,\right\| \,=\, 0\; \;\text{ and it is called a Cauchy sequence if}\]
$\lim\limits_{l \,,\, k \,\to\, \infty}\,\left \|\,x_{l} \,-\, x_{\,k} \,,\, x_{\,2} \,,\, \cdots \,,\, x_{\,n}\,\right\| \,=\, 0$
for every \,$ x_{\,2},\, \cdots,\, x_{\,n} \,\in\, X$.\;The space \,$X$\, is said to be complete if every Cauchy sequence in this space is convergent in \,$X$.\;A n-inner product space is called n-Hilbert space if it is complete with respect to its induce norm.
\end{definition}

\begin{note}\cite{Prasenjit}
Let \,$L_{F}$\, denote the linear subspace of \,$X$\, spanned by the non-empty finite set \,$F \,=\, \left\{\,\,a_{\,2} \,,\, a_{\,3} \,,\, \cdots \,,\, a_{\,n}\,\right\}$, where \,$a_{\,2},\, a_{\,3},\, \cdots,\, a_{\,n}$\, are fixed elements in \,$X$.\;Then the quotient space \,$ X \,/\, L_{F}$\, is a normed linear space with respect to the norm, \,$\left\|\,x \,+\, L_{F}\,\right\|_{F} \,=\, \left\|\,x \,,\, a_{\,2} \,,\,  \cdots \,,\, a_{\,n}\,\right\|$\, for every \,$x \,\in\, X$.
Let \,$M_{F}$\, be the algebraic complement of \,$L_{F}$, then \,$X \,=\, L_{F} \,\oplus\, M_{F}$.\;Define \,$\left<\,x \,,\, y\,\right>_{F} \,=\, \left<\,x \,,\, y \;|\; a_{\,2} \,,\,  \cdots \,,\, a_{\,n}\,\right>\; \;\text{on}\; \;X$.\;Then \,$\left<\,\cdot \,,\, \cdot\,\right>_{F}$\, is a semi-inner product on \,$X$\, and this semi-inner product induces an inner product on \,$X \,/\, L_{F}$\, which is given by
\[\left<\,x \,+\, L_{F} \,,\, y \,+\, L_{F}\,\right>_{F} \,=\, \left<\,x \,,\, y\,\right>_{F} \,=\, \left<\,x \,,\, y \,|\, a_{\,2} \,,\,  \cdots \,,\, a_{\,n} \,\right>\;\; \;\forall \;\; x,\, y \,\in\, X.\]
By identifying \,$ X \,/\, L_{F}$\; with \,$M_{F}$\; in an obvious way, we obtain an inner product on \,$M_{F}$.\;Now for every \,$x \,\in\, M_{F}$, we define \,$\|\,x\,\|_{F} \;=\; \sqrt{\left<\,x \,,\, x \,\right>_{F}}$\, and it can be easily verify that \,$\left(\,M_{F} \,,\, \|\,\cdot\,\|_{F}\,\right)$\; is a norm space.\;Let \,$X_{F}$\; be the completion of the inner product space \,$M_{F}$.
\end{note}

For the remaining part of this paper, \,$\left(\,X \,,\, \left<\,\cdot \,,\, \cdot \,|\, \cdot \,,\, \cdots \,,\, \cdot\,\right> \,\right)$\; is consider to be a \,$n$-Hilbert space and \,$I$\, will denote the identity operator on \,$X_{F}$. 

\begin{definition}\label{def0.1}\cite{Prasenjit}
A sequence \,$\left\{\,f_{\,i}\,\right\}^{\,\infty}_{\,i \,=\, 1}$\, in \,$X$\, is said to be a frame associated to \,$\left(\,a_{\,2},\, \cdots,\, a_{\,n}\,\right)$\, for \,$X$\, if there exist constants \,$0 \,<\, A \,\leq\, B \,<\, \infty$\, such that 
\[ A \, \left\|\,f \,,\, a_{\,2} \,,\, \cdots \,,\, a_{\,n} \,\right\|^{\,2} \,\leq\, \sum\limits^{\infty}_{i \,=\, 1}\,\left|\,\left<\,f \,,\, f_{\,i} \,|\, a_{\,2} \,,\, \cdots \,,\, a_{\,n}\,\right>\,\right|^{\,2} \,\leq\, B\, \left\|\,f \,,\, a_{\,2} \,,\, \cdots \,,\, a_{\,n}\,\right\|^{\,2}\]
for all \,$f \,\in\, X$.\;The constants \,$A$\, and \,$B$\, are called the frame bounds.\;If the sequence \,$\left\{\,f_{\,i}\,\right\}^{\,\infty}_{\,i \,=\, 1}$\; satisfies the inequality 
\[\sum\limits^{\infty}_{i \,=\, 1}\,\left|\,\left<\,f \,,\, f_{\,i} \,|\, a_{\,2} \,,\, \cdots \,,\, a_{\,n}\,\right>\,\right|^{\,2} \,\leq\, B\; \left\|\,f \,,\, a_{\,2} \,,\, \cdots \,,\, a_{\,n}\,\right\|^{\,2}\; \;\forall\; f \,\in\, X\] is called a Bessel sequence associated to \,$\left(\,a_{\,2},\, \cdots,\, a_{\,n}\,\right)$\, in \,$X$\, with bound \,$B$.
\end{definition}

\begin{theorem}\label{th2}\cite{Prasenjit}
Let \,$\left\{\,f_{\,i}\,\right\}^{\,\infty}_{\,i \,=\, 1}$\, be a sequence in \,$X$.\,Then \,$\left\{\,f_{\,i}\,\right\}^{\,\infty}_{\,i \,=\, 1}$\, is a frame associated to \,$\left(\,a_{\,2},\, \cdots,\, a_{\,n}\,\right)$\; with bounds \,$A \;\;\&\;\; B$\; if and only if it is a frame for the Hilbert space \,$X_{F}$\; with bounds \,$A \;\;\&\;\; B$.
\end{theorem}

\begin{definition}\cite{Prasenjit}
Let \,$\left\{\,f_{\,i}\,\right\}_{i \,=\, 1}^{\infty}$\; be a Bessel sequence associated to \,$\left(\,a_{\,2},\, \cdots,\, a_{\,n}\,\right)$\, for \,$X$.\;Then the bounded linear operator \,$ T \,:\, l^{\,2}\,(\,\mathbb{N}\,) \,\to\, X_{F}$, defined by \,$T\,\{\,c_{i}\,\} \,=\, \sum\limits_{i \,=\, 1}^{\infty} \;c_{\,i}\,f_{\,i}$, is called  pre-frame operator and its adjoint operator \,$T\,^{\,\ast} \,:\, X_{F} \,\to\, l^{\,2}\,(\,\mathbb{N}\,)$, given by \;$T^{\,\ast}\,(\,f\,) \,=\, \left \{\, \left <\,f \,,\, f_{i} \,|\, a_{\,2},\, \cdots,\, a_{\,n}\,\right >\,\right \}_{i \,=\, 1}^{\infty} \;\forall\, f \,\in\, X_{F}$\; is called the analysis operator\,.\;The operator \,$S \,:\, X_{F} \,\to\, X_{F}$\; defined by \,$S_{F}\,(\,f\,) \,=\, T\,T^{\,\ast}\,(\,f\,) \,=\, \sum\limits^{\infty}_{i \,=\, 1}\; \left <\,f \,,\, f_{\,i} \,|\, a_{\,2},\, \cdots,\, a_{\,n} \, \right >\,f_{\,i}\, \;\forall\, f \,\in\, X_{F}$\; is called the frame operator.
\end{definition}

\begin{note}\cite{Prasenjit}
If \,$\left\{\,f_{\,i}\,\right\}_{i \,=\, 1}^{\infty}$\; is a frame associated to \,$\left(\,a_{\,2},\, \cdots,\, a_{\,n}\,\right)$\, for \,$X$, then the frame operator \,$S_{F}$\; is bounded, positive, self-adjoint and invertible.
\end{note}

\section{$K$-frame in $n$-Hilbert space}

\begin{definition}\label{def1}
Let \,$K$\, be a bounded linear operator on \,$X_{F}$.\,Then a sequence \,$\{\,f_{\,i}\,\}_{i \,=\, 1}^{\infty} \,\subseteq\, X$\, is said to be a \,$K$-frame associated to \,$\left(\,a_{\,2},\, \cdots,\, a_{\,n}\,\right)$\, for \,$X$\, if there exist constants \,$0 \,<\, A \,\leq\, B \,<\, \infty$\, such that
\[A\,\left \|\,K^{\,\ast}\,f \,,\, a_{\,2} \,,\, \cdots \,,\, a_{\,n}\,\right \|^{\,2} \,\leq\, \sum\limits_{i \,=\, 1}^{\infty}\, \left|\,\left <\,f \,,\,  f_{\,i} \,|\, a_{\,2} \,,\, \cdots \,,\, a_{\,n}\,\right >\,\right|^{\,2} \,\leq\, B\,\left\|\,f \,,\, a_{\,2} \,,\, \cdots \,,\, a_{\,n}\,\right\|^{\,2},\]
for all \,$f \,\in\, X_{F}$.\,In particular, if \,$K \,=\, I$, then by Theorem (\ref{th2}), \,$\{\,f_{\,i}\,\}_{i \,=\, 1}^{\infty}$\, is a frame associated to \,$\left(\,a_{\,2},\, \cdots,\, a_{\,n}\,\right)$\, for \,$X$.\,Obviously, every \,$K$-frame associated to \,$\left(\,a_{\,2},\, \cdots,\, a_{\,n}\,\right)$\, is a Bessel sequence associated to \,$\left(\,a_{\,2},\, \cdots,\, a_{\,n}\,\right)$\, in \,$X$.  
\end{definition}

\begin{note}
In general, the frame operator of a \,$K$-frame associated to \,$\left(\,a_{\,2},\, \cdots,\, a_{\,n}\,\right)$\, is not invertible.\,But, if \,$K \,\in\, \mathcal{B}\,(\,X_{F}\,)$\, has closed range, then \,$S_{F} \,:\, \mathcal{R}\,(\,K\,) \,\to\, S_{F}\,\left(\,\mathcal{R}\,(\,K\,)\,\right)$\, is an invertible operator.\,For \,$f \,\in\, \mathcal{R}\,(\,K\,)$, we have
\[\left\|\,f,\, a_{\,2},\, \cdots,\, a_{\,n}\,\right\|^{\,2} \,=\, \left\|\,(\,K^{\,\dagger}\,)^{\,\ast}\,K^{\,\ast}\,f,\, a_{\,2},\, \cdots,\, a_{\,n}\,\right\|^{\,2} \,\leq\, \left\|\,K^{\,\dagger}\,\right\|^{\,2}\,\left\|\,K^{\,\ast}\,f,\, a_{\,2},\, \cdots,\, a_{\,n}\,\right\|^{\,2}.\]
Therefore, if \,$\{\,f_{\,i}\,\}_{i \,=\, 1}^{\infty}$\, is a \,$K$-frame associated to \,$\left(\,a_{\,2},\, \cdots,\, a_{\,n}\,\right)$\, for \,$X$, by definition (\ref{def1}), we get
\[A\, \left\|\,K^{\,\dagger}\,\right\|^{\,-\, 2}\,\left\|\,f,\, a_{\,2},\, \cdots,\, a_{\,n}\,\right\|^{\,2} \,\leq\, \left<\,S_{F}\,f,\, f \,|\, a_{\,2},\, \cdots,\, a_{\,n}\,\right>\,\leq\, B\,\left\|\,f,\, a_{\,2},\, \cdots,\, a_{\,n}\,\right\|^{\,2}\]
and furthermore for each \,$f \,\in\, S_{F}\,\left(\,\mathcal{R}\,(\,K\,)\,\right)$, we have
\[B^{\,-\, 1}\left\|\,f,\, a_{\,2},\, \cdots,\, a_{\,n}\,\right\|^{\,2} \leq \left<\,S^{\,-\, 1}_{F}\,f,\, f \,|\, a_{\,2},\, \cdots,\, a_{\,n}\,\right> \leq A^{\,-\, 1}\left\|\,K^{\,\dagger}\,\right\|^{\,2}\left\|\,f,\, a_{\,2},\, \cdots,\, a_{\,n}\,\right\|^{\,2}\] 
\end{note}

\begin{theorem}\label{thm2}
Let \,$\left \{\,f_{\,i}\, \right \}^{\infty}_{i \,=\, 1}$\, be a K-frame associated to \,$\left(\,a_{\,2},\, \cdots,\, a_{\,n}\,\right)$\, for \,$X$\, and \,$T \,\in\, \mathcal{B}\,\left(\,X_{F}\,\right)$\, with \,$\mathcal{R}\,(\,T\,) \,\subset\, \mathcal{R}\,(\,K\,)$.\;Then \,$\left \{\,f_{\,i}\, \right \}^{\infty}_{i \,=\, 1}$\; is a T-frame associated to \,$\left(\,a_{\,2},\, \cdots,\, a_{\,n}\,\right)$\, for \,$X$.
\end{theorem}

\begin{proof} 
Suppose \,$\left \{\,f_{\,i}\,\right \}^{\infty}_{i \,=\, 1}$\, is a \,$K$-frame associated to \,$\left(\,a_{\,2},\, \cdots,\, a_{\,n}\,\right)$\, for \,$X$.\;Then for each \,$f \,\in\, X_{F}$, there exist constants \,$A,\, B \,>\, 0$\; such that
\[A\,\left \|\,K^{\,\ast}\,f \,,\, a_{\,2} \,,\, \cdots \,,\, a_{\,n}\,\right \|^{\,2} \,\leq\, \sum\limits_{i \,=\, 1}^{\infty}\, \left|\,\left <\,f \,,\,  f_{\,i} \,|\, a_{\,2} \,,\, \cdots \,,\, a_{\,n}\,\right >\,\right|^{\,2} \,\leq\, B\,\left \|\,f \,,\, a_{\,2} \,,\, \cdots \,,\, a_{\,n}\,\right \|^{\,2}.\]
Since \,$R\,(\,T\,) \,\subset\, R\,(\,K\,)$, by Theorem (\ref{th1}), \,$\exists\; \lambda \,>\, 0$\, such that \,$T\,T^{\,\ast} \,\leq\, \lambda^{\,2}\, K\,K^{\,\ast}$.\;Thus,
\[\dfrac{A}{\lambda^{\,2}}\,\left\|\,T^{\,\ast}\,f \,,\, a_{\,2} \,,\, \cdots \,,\, a_{\,n}\,\right\|^{\,2} \,=\, \dfrac{A}{\lambda^{\,2}}\,\left<\,T\,T^{\,\ast}\,f \,,\, f \,|\, a_{\,2} \,,\, \cdots \,,\, a_{\,n}\,\right>\]
\[\hspace{3cm} \,=\, \left<\,\dfrac{A}{\lambda^{\,2}}\,T\,T^{\,\ast}\,f \,,\, f \,|\, a_{\,2} \,,\, \cdots \,,\, a_{\,n}\,\right> \,\leq\, \left<\,A\;K\,K^{\,\ast}\,f \,,\, f \,|\, a_{\,2} \,,\, \cdots \,,\, a_{\,n} \,\right>\]
\[\hspace{2cm} \,=\, A\,\left \|\,K^{\,\ast}\,f \,,\, a_{\,2} \,,\, \cdots \,,\, a_{\,n}\,\right \|^{\,2}.\;\text{Therefore, for each \,$f \,\in\, X_{F}$,}\]
\[\dfrac{A}{\lambda^{\,2}}\left\|\,T^{\,\ast}\,f \,,\, a_{\,2} \,,\, \cdots \,,\, a_{\,n}\,\right\|^{\,2} \;\leq\; \sum\limits^{\infty}_{i \,=\, 1} \left |\,\left <\,f \,,\, f_{\,i} \,|\,  a_{\,2} \,,\, \cdots \,,\, a_{\,n} \,\right >\,\right |^{\,2} \,\leq\, B\,\left\|\,f \,,\, a_{\,2} \,,\, \cdots \,,\, a_{\,n}\,\right\|^{\,2}.\]
Hence, \,$\left \{\,f_{\,i}\,\right \}^{\infty}_{i \,=\, 1}$\; is a \,$T$-frame associated to \,$\left(\,a_{\,2},\, \cdots,\, a_{\,n}\,\right)$\, for \,$X$.
\end{proof}

\begin{theorem}\label{th6}
Let \,$\left \{\,f_{\,i}\, \right \}^{\infty}_{i \,=\, 1}$\, be a K-frame associated to \,$\left(\,a_{\,2},\, \cdots,\, a_{\,n}\,\right)$\, for \,$X$\, with bounds \,$A,\,B$\, and \,$T \,\in\, \mathcal{B}\,(\,X_{F}\,)$\, be an invertible with \,$T\,K \,=\, K\,T$, then \,$\left \{\,T\,f_{\,i}\,\right \}^{\infty}_{i \,=\, 1}$\, is a K-frame associated to \,$\left(\,a_{\,2},\, \cdots,\, a_{\,n}\,\right)$\, for \,$X$.
\end{theorem}

\begin{proof}
Since \,$T$\, is invertible, for each \,$f \,\in\, X_{F}$, 
\[ \left\|\,K^{\,\ast}\,f \,,\, a_{\,2} \,,\, \cdots \,,\, a_{\,n} \,\right\|^{\,2} \,=\, \left\|\,\left (\,T^{\,-\, 1}\,\right )^{\,\ast}\,T^{\,\ast}\,K^{\,\ast}\,f \,,\, a_{\,2} \,,\, \cdots \,,\, a_{\,n}\,\right\|^{\,2}\]
\[\hspace{5cm} \,\leq\, \left\|\,\left(\,T^{\,-\, 1}\,\right )^{\,\ast}\,\right\|^{\,2} \,\left\|\,T^{\,\ast}\,K^{\,\ast}\,f \,,\, a_{\,2} \,,\, \cdots \,,\, a_{\,n}\,\right \|^{\,2}.\]
\begin{equation}\label{eq2}
\Rightarrow\; \left \|\,T^{\,-\, 1}\,\right\|^{\,-\, 2}\,\left\|\,K^{\,\ast}\,f \,,\, a_{\,2} \,,\, \cdots \,,\, a_{\,n}\,\right\|^{\,2} \,\leq\, \left \|\,T^{\,\ast}\,K^{\,\ast}\,f \,,\, a_{\,2} \,,\, \cdots \,,\, a_{\,n}\,\right\|^{\,2}.
\end{equation}
Also, since \,$\{\,f_{\,i}\,\}^{\infty}_{i \,=\, 1}$\, is a \,$K$-frame associated to \,$\left(\,a_{\,2},\, \cdots,\, a_{\,n}\,\right)$, for \,$f \,\in\, X_{F}$,  
\[\sum\limits^{\infty}_{i \,=\, 1} \,\left|\,\left <\,f \,,\, T\,f_{\,i} \,|\, a_{\,2} \,,\, \cdots \,,\, a_{\,n}\,\right >\,\right |^{\,2} \,=\, \sum\limits^{\infty}_{i \,=\, 1}\,\left|\,\left <\,T^{\,\ast}\,f \,,\, f_{\,i} \,|\, a_{\,2} \,,\, \cdots \,,\, a_{\,n}\,\right >\,\right|^{\,2}\]
\[ \,\geq\, A\, \left\|\,K^{\,\ast}\,T^{\,\ast}\,f \,,\, a_{\,2} \,,\, \cdots \,,\, a_{\,n}\,\right\|^{\,2} \,=\, A\,\left\|\,T^{\,\ast}\,K^{\,\ast}\,f \,,\, a_{\,2} \,,\, \cdots \,,\, a_{\,n} \,\right\|^{\,2}\; \left[\;\text{since}\; T\,K \,=\, K\,T\,\right]\]
\[ \geq\, A\;\left \|\,T^{\,-\, 1}\,\right\|^{\,-2}\,\left\|\, K^{\,\ast}\,f \,,\, a_{\,2} \,,\, \cdots \,,\, a_{\,n}\,\right\|^{\,2}\; \;[\;\text{using} \;(\ref{eq2})].\hspace{5cm}\]
On the other hand, for all \,$f \,\in\, X_{F}$,
\[\sum\limits^{\infty}_{i \,=\, 1}\,\left|\,\left <\,f \,,\, T\,f_{\,i} \,|\, a_{\,2} \,,\, \cdots \,,\, a_{\,n}\,\right >\,\right |^{\,2} \,=\, \sum\limits^{\infty}_{i \,=\, 1}\,\left|\,\left <\,T^{\,\ast}\,f \,,\, f_{\,i} \,|\, a_{\,2} \,,\, \cdots \,,\, a_{\,n}\,\right >\,\right|^{\,2}\hspace{1.5cm}\]
\[\hspace{2.3cm} \,\leq\, B\, \left \|\,T^{\,\ast}\,f \,,\, a_{\,2} \,,\, \cdots \,,\, a_{\,n}\,\right\|^{\,2} \,\leq\, B \,\left\|\,T\, \right \|^{\,2}\;\|\,f \,,\, a_{\,2} \,,\, \cdots \,,\, a_{\,n}\,\|^{\,2}.\]
Hence, \,$\left \{\,T\,f_{\,i}\,\right \}^{\infty}_{i \,=\, 1}$\; is a \,$K$-frame associated to \,$\left(\,a_{\,2},\, \cdots,\, a_{\,n}\,\right)$\, for \,$X$.
\end{proof}

\begin{theorem}
Let \,$\left \{\,f_{\,i}\, \right \}^{\infty}_{i \,=\, 1}$\, be a K-frame associated to \,$\left(\,a_{\,2},\, \cdots,\, a_{\,n}\,\right)$\, for \,$X$\, with bounds \,$A,\,B$\, and \,$T \,\in\, \mathcal{B}\,(\,X_{F}\,)$\, such that \,$T\,T^{\,\ast} \,=\, I$\, with \,$T \,K \,=\, K \,T$.\;Then \,$\left \{\,T\,f_{\,i}\,\right \}^{\infty}_{i \,=\, 1}$\, is a K-frame associated to \,$\left(\,a_{\,2},\, \cdots,\, a_{\,n}\,\right)$\, for \,$X$.
\end{theorem}

\begin{proof} 
Since \,$T\,T^{\,\ast} \,=\, I$, for \,$f \,\in\, X_{F}$, \,$\left\|\,T^{\,\ast}\,f \,,\, a_{\,2} \,,\, \cdots \,,\, a_{\,n}\,\right \|^{\,2} \,=\, \|\,f \,,\, a_{\,2} \,,\, \cdots \,,\, a_{\,n}\,\|^{\,2}$\, and this implies that \,$\left\|\,T^{\,\ast}\,K^{\,\ast}\,f \,,\, a_{\,2} \,,\, \cdots \,,\, a_{\,n} \,\right\|^{\,2} \,=\, \|\,K^{\,\ast}\,f \,,\, a_{\,2} \,,\, \cdots \,,\, a_{\,n}\,\|^{\,2}$.\;Also, since \,$\{\,f_{\,i}\,\}^{\infty}_{i \,=\, 1}$\, is a \,$K$-frame associated to \,$\left(\,a_{\,2},\, \cdots,\, a_{\,n}\,\right)$, for each \,$f \,\in\, X_{F}$, 
\[\sum\limits^{\infty}_{i \,=\, 1}\,\left|\,\left <\,f \,,\, T\, f_{\,i} \,|\, a_{\,2} \,,\, \cdots \,,\, a_{\,n}\,\right >\,\right|^{\,2} \,=\, \sum\limits^{\infty}_{i \,=\, 1}\,\left|\,\left <\,T^{\,\ast}\,f \,,\, f_{\,i} \,|\, a_{\,2} \,,\, \cdots \,,\, a_{\,n}\,\right >\,\right|^{\,2}\hspace{1.5cm}\]
\[\hspace{3cm} \,\geq\, A\, \left\|\,K^{\,\ast}\,T^{\,\ast}\,f \,,\, a_{\,2} \,,\, \cdots \,,\, a_{\,n}\,\right \|^{\,2} \,=\, A\, \left\|\,T^{\,\ast}\,K^{\,\ast}\,f \,,\, a_{\,2} \,,\, \cdots \,,\, a_{\,n} \,\right\|^{\,2}\]
\[ \,=\, A\, \left\|\,K^{\,\ast}\,f \,,\, a_{\,2} \,,\, \cdots \,,\, a_{\,n}\,\right\|^{\,2}.\hspace{2.9cm}\]
Thus, we see that \,$\left \{\,T\,f_{\,i}\,\right \}^{\infty}_{i \,=\, 1}$\, satisfies lower \,$K$-frame condition.\;Following the proof of the Theorem (\ref{th6}), it can be shown that it also satisfies upper \,$K$-frame condition and therefore it is a \,$K$-frame associated to \,$\left(\,a_{\,2},\, \cdots,\, a_{\,n}\,\right)$\, for \,$X$.
\end{proof}

\begin{theorem}\label{th7}
Let \,$\left \{\,f_{\,i}\,\right \}^{\infty}_{i \,=\, 1}$\; be a sequence in \,$X$.\;Then \,$\left \{\,f_{\,i}\,\right \}^{\infty}_{i \,=\, 1}$\, is a K-frame associated to \,$\left(\,a_{\,2},\, \cdots,\, a_{\,n}\,\right)$\, for \,$X$\, if and only if there exists a bounded linear operator \,$T \,:\, l^{\,2}\,(\,\mathbb{N}\,) \,\to\, X_{F}$\, such that \,$f_{\,i} \,=\, T\,e_{\,i}$\; and \,${\mathcal{R}}\,(\,K\,) \,\subset\, {\,\mathcal R}\,(\,T\,)$, where \,$\{\,e_{\,i}\,\}^{\infty}_{i \,=\, 1}$\; is an \,$F$-orthonormal basis for \,$l^{\,2}\,(\,\mathbb{N}\,)$.
\end{theorem}

\begin{proof}
First we suppose that \,$\{\,f_{\,i}\,\}^{\infty}_{i \,=\, 1}$\, is a \,$K$-frame associated to \,$\left(\,a_{\,2},\, \cdots,\, a_{\,n}\,\right)$.\\Then, for each \,$f \,\in\, X_{F}$, there exist constants \,$A,\, B \,>\, 0$\; such that
\[A\,\left \|\,K^{\,\ast}\,f \,,\, a_{\,2} \,,\, \cdots \,,\, a_{\,n}\,\right \|^{\,2} \,\leq\, \sum\limits_{i \,=\, 1}^{\infty}\, \left|\,\left <\,f \,,\,  f_{\,i} \,|\, a_{\,2} \,,\, \cdots \,,\, a_{\,n}\,\right >\,\right|^{\,2} \,\leq\, B\,\left \|\,f \,,\, a_{\,2} \,,\, \cdots \,,\, a_{\,n}\,\right \|^{\,2}.\]
Now, we consider the linear operator \,$L \,:\, X_{F} \,\to\, l^{\,2}\,(\,\mathbb{N}\,)$\; defined by
\[ L\,(\,f\,) \,=\, \sum\limits^{\infty}_{i \,=\, 1}\,\left <\,f \,,\, f_{\,i} \,|\, a_{\,2} \,,\, \cdots \,,\, a_{\,n}\,\right > \,e_{\,i}\; \;\forall\; f \,\in\, X_{F}.\]
Since \,$\{\,e_{\,i}\,\}^{\infty}_{i \,=\, 1}$\; is an \,$F$-orthonormal basis for \,$l^{\,2}\,(\,\mathbb{N}\,)$, we can write
\[\left\|\,L\,(\,f\,)\,\right\|_{\,l^{\,2}}^{\,2} \,=\, \sum\limits^{\infty}_{i \,=\, 1}\,\left|\,\left <\,f \,,\, f_{\,i} \,|\,  a_{\,2} \,,\, \cdots \,,\, a_{\,n}\,\right >\,\right|^{\,2} \,\leq\, B\, \|\,f \,,\, a_{\,2} \,,\, \cdots \,,\, a_{\,n}\,\|^{\,2}.\]
Thus, \,$L$\, is well-defined and bounded linear operator on \,$X_{F}$.\;So, the adjoint operator \,$L^{\,\ast} \,:\, l^{\,2}\,(\,\mathbb{N}\,) \,\to\, X_{F}$\, exists and then for each \,$f \,\in\, X_{F}$, we get 
\[\left <\,L^{\,\ast}\,e_{\,i}\, \,,\, f \,|\, a_{\,2} \,,\, \cdots \,,\, a_{\,n}\,\right > \,=\, \left <\,e_{\,i}\, \,,\, L\,(\,f\,) \;|\; a_{\,2} \,,\, \cdots \,,\, a_{\,n}\,\right>\hspace{4cm}\]
\[\hspace{2cm} \;=\, \left <\,e_{\,i}\, \,,\, \sum\limits^{\infty}_{i \,=\, 1}\,\left <\,f \,,\, f_{\,i} \,|\, a_{\,2} \,,\, \cdots \,,\, a_{\,n}\,\right > \,e_{\,i} \,|\, a_{\,2} \,,\, \cdots \,,\, a_{\,n}\,\right >\]
\[\hspace{1.56cm}\,=\, \overline{\left <\,f \,,\, f_{\,i} \,|\, a_{\,2} \,,\, \cdots \,,\, a_{\,n}\;\right >} \,=\, \left <\,f_{\,i} \,,\, f \,|\, a_{\,2} \,,\, \cdots \,,\, a_{\,n}\,\right >.\]
The above calculation shows that, \,$L^{\,\ast}\,(\,e_{\,i}\,) \,=\, f_{\,i}$.\;Also, from the definition (\ref{def1}), we get \,$A\, \|\,K^{\,\ast}\,f\,\|_{F}^{\,2} \,\leq\, \left\|\,L\,(\,f\,)\,\right \|_{\,l^{\,2}}^{\,2}$\, and this implies that
\[\left<\,A\,K\,K^{\,\ast}\,f \,,\, f \,|\, a_{\,2} \,,\, \cdots \,,\, a_{\,n}\,\right> \,\leq\, \left<\,L^{\,\ast}\,L\,f \,,\, f \,|\, a_{\,2} \,,\, \cdots \,,\, a_{\,n}\,\right>\]
\[ \Rightarrow\; A \,K\,K^{\,\ast} \,\leq\, T\,T^{\,\ast},\; \text{where}\;\; T \,=\, L^{\,\ast}\hspace{3.9cm}\] and hence from the  Theorem (\ref{th1}), \,${\,\mathcal R}\,(\,K\,) \,\subset\, {\,\mathcal R}\,(\,T\,)$.\\

\text{Conversely}, suppose that \,$T \,:\, l^{\,2}\,(\,\mathbb{N}\,) \,\to\, X_{F}$\; be a bounded linear operator such that \,$f_{\,i} \,=\, T\,e_{\,i}$\; and \,${\,\mathcal R}\,(\,K\,) \,\subset\, {\,\mathcal R}\,(\,T\,)$.\;We have to show that \,$\{\,f_{\,i}\,\}_{i \,=\, 1}^{\infty}$\; is a \,$K$-frame associated to \,$\left(\,a_{\,2},\, \cdots,\, a_{\,n}\,\right)$.\;Let \,$g \,\in\, l^{\,2}\,(\,\mathbb{N}\,)$\, then \,$g \,=\, \sum\limits^{\infty}_{i \,=\, 1}\,c_{\,i}\,e_{\,i}$, where \,$c_{\,i} \,=\, \left<\,g \,,\, e_{\,i} \,|\, a_{\,2} \,,\, \cdots \,,\, a_{\,n}\,\right>$.\;Now, for all \,$g \,\in\, l^{\,2}\,(\,\mathbb{N}\,)$, we have 
\[\left<\,T^{\,\ast}\,f,\, g \,|\, a_{\,2},\, \cdots,\, a_{\,n}\,\right> \,=\, \left<\,T^{\,\ast}\,f,\, \sum\limits^{\infty}_{i \,=\, 1}\,c_{\,i}\,e_{\,i} \,|\, a_{\,2},\, \cdots,\, a_{\,n}\,\right>\]
\[\hspace{2cm} \,=\, \sum\limits^{\infty}_{i \,=\, 1}\, \overline{c_{\,i}}\,\left<\,f,\, T\,e_{\,i} \,|\, a_{\,2},\, \cdots,\, a_{\,n}\,\right>\,=\, \sum\limits^{\infty}_{i \,=\, 1}\,\overline{\,c_{\,i}}\,\left<\,f \,,\, f_{\,i} \,|\, a_{\,2} \,,\, \cdots \,,\, a_{\,n}\,\right>\] 
\[\,=\, \sum\limits^{\infty}_{i \,=\, 1}\,\overline{\left<\,g \,,\, e_{\,i} \,|\, a_{\,2} \,,\, \cdots \,,\, a_{\,n}\,\right>}\,\left<\,f \,,\, f_{\,i} \,|\, a_{\,2} \,,\, \cdots \,,\, a_{\,n}\,\right>\hspace{.4cm}\]
\[ \,=\, \sum\limits^{\infty}_{i \,=\, 1}\,\left<\,e_{\,i} \,,\, g  \,|\, a_{\,2} \,,\, \cdots \,,\, a_{\,n}\,\right> \,\left<\,f \,,\, f_{\,i} \,|\, a_{\,2} \,,\, \cdots \,,\, a_{\,n}\,\right>\hspace{.4cm}\]
\[\hspace{.3cm} \;=\; \left<\,\sum\limits^{\infty}_{i \,=\, 1}\,\left<\,f \,,\, f_{\,i} \,|\, a_{\,2} \,,\, \cdots \,,\, a_{\,n}\,\right>\,e_{\,i} \;,\; g  \,|\, a_{\,2} \,,\, \cdots \,,\, a_{\,n}\,\right>.\] 
\[\Rightarrow\, T^{\,\ast}\,(\,f\,) \;=\; \sum\limits^{\infty}_{i \,=\, 1}\,\left<\,f \,,\, f_{\,i} \,|\, a_{\,2} \,,\, \cdots \,,\, a_{\,n}\,\right>\,e_{\,i}\; \;\forall\; f \,\in\, X_{F}.\]
Thus, for all \,$f \,\in\, X_{F}$, \,$\sum\limits^{\infty}_{i \,=\, 1}\,\left|\,\left <\,f \,,\, f_{\,i} \,|\,  a_{\,2} \,,\, \cdots \,,\, a_{\,n}\,\right >\,\right |^{\,2}$
\[\,=\, \sum\limits^{\infty}_{i \,=\, 1}\,\left |\,\left <\,f \,,\, T\,e_{\,i} \,|\, a_{\,2} \,,\, \cdots \,,\, a_{\,n}\,\right >\,\right |^{\,2} \,=\, \sum\limits^{\infty}_{i \,=\, 1}\,\left|\,\left <\,T^{\,\ast}\,f \,,\, e_{\,i} \,|\,  a_{\,2} \,,\, \cdots \,,\, a_{\,n}\,\right >\,\right |^{\,2}\]
\[\hspace{.5cm} \,=\, \left \|\,T^{\,\ast}\,f \,,\,  a_{\,2} \,,\, \cdots \,,\, a_{\,n}\,\right \|^{\,2}\;[\;\text{since \,$\{\,e_{\,i}\,\}^{\infty}_{i \,=\, 1}$\; is an \,$F$-orthonormal basis}\;]\]
\[\,\leq\, \left \|\,T^{\,\ast}\,\right \|^{\,2}\,\|\,f \,,\, a_{\,2} \,,\, \cdots \,,\, a_{\,n}\,\|^{\,2} \,=\, \left \|\,T\,\right \|^{\,2}\,\|\,f \,,\, a_{\,2} \,,\, \cdots \,,\, a_{\,n}\,\|^{\,2}\hspace{1.5cm}\] 
\[\Rightarrow\, \sum\limits^{\infty}_{i \,=\, 1}\,\left |\,\left <\,f \,,\, f_{\,i} \,|\, a_{\,2} \,,\, \cdots \,,\, a_{\,n}\,\right >\,\right |^{\,2} \,\leq\, \left \|\,T\,\right \|^{\,2}\, \|\,f \,,\, a_{\,2} \,,\, \cdots \,,\, a_{\,n} \,\|^{\,2}\; \;\forall\; f \,\in\, X_{F}.\]
Thus, \,$\{\,f_{\,i}\,\}^{\infty}_{i \,=\, 1}$\; is a Bessel sequence associated to \,$\left(\,a_{\,2},\, \cdots,\, a_{\,n}\,\right)$.\;Since \,${\,\mathcal R}\,(\,K\,) \,\subset\, {\,\mathcal R}\,(\,T\,)$, from Theorem (\ref{th1}), there exists \,$A \,>\, 0$\; such that \,$A\,K\,K^{\,\ast} \,\leq\, T\,T^{\,\ast}$.\;Hence following the proof of the Theorem (\ref{thm2}), for all \,$f \,\in\, X_{F}$ 
\[A\,\|\,K^{\,\ast}\,f,\, a_{\,2},\, \cdots,\, a_{\,n}\,\|^{\,2} \,\leq\, \left\|\,T^{\,\ast}\,f,\, a_{\,2},\, \cdots,\, a_{\,n}\,\right\|^{\,2} \,=\, \sum\limits^{\infty}_{i \,=\, 1}\,\left | \,\left <\,f \,,\, f_{\,i} \,|\, a_{\,2},\, \cdots, a_{\,n}\,\right >\,\right |^{\,2}.\]
Hence, \,$\{\,f_{\,i}\,\}^{\infty}_{i \,=\, 1}$\; is a \,$K$-frame associated to \,$\left(\,a_{\,2},\, \cdots,\, a_{\,n}\,\right)$\, for \,$X$.
\end{proof}

\begin{theorem}\label{th8}
Let \,$\{\,f_{\,i}\,\}^{\infty}_{i=1}\, \;\text{and}\; \,\{\,g_{\,i}\,\}^{\infty}_{i \,=\, 1}$\; be K-frames associated to \,$\left(\,a_{\,2},\, \cdots,\, a_{\,n}\,\right)$\, for \,$X$\; with the corresponding pre frame operators \,$T$\, and \,$L$, respectively.\;If \,$T\,L^{\,\ast}$\, and \,$L\,T^{\,\ast}$\, are positive operators, then  \,$\{\,f_{\,i} \,+\, g_{\,i}\,\}^{\infty}_{i \,=\, 1}$\, is also a K-frame associated to \,$\left(\,a_{\,2},\, \cdots,\, a_{\,n}\,\right)$\, for \,$X$.
\end{theorem}
\begin{proof} 
Let \,$\{\,f_{\,i}\,\}^{\infty}_{i \,=\, 1}\, \;\text{and}\; \,\{\,g_{\,i}\,\}^{\infty}_{i \,=\, 1}$\; be two \,$K$-frames associated to \,$\left(\,a_{\,2},\, \cdots,\, a_{\,n}\,\right)$\, for \,$X$.\;Then by Theorem (\ref{th7}), there exist bounded linear operators \,$T\, \;\text{and}\; \,L$\; such that \,$T\,e_{\,i} \,=\, f_{\,i},\; L\,e_{\,i} \,=\, g_{\,i}\; \;\text{and}\; \,{\mathcal R}\,(\,K\,) \,\subset\, {\mathcal R}\,(\,T\,),\; {\mathcal R}\,(\,K\,) \,\subset\, {\mathcal R}\,(\,L\,)$, where \,$\{\,e_{\,i}\,\}^{\infty}_{i \,=\, 1}$\; is an \,$F$-orthonormal basis for \,$l^{\,2}\,(\,\mathbb{N}\,)$.\;Now, we have \,${\mathcal R}\,(\,K\,) \,\subset\, {\mathcal R}\,(\,T\,) \,+\, {\mathcal R}\,(\,L\,)$.\;By Theorem (\ref{thm1}), \,$K\,K^{\,\ast} \,\leq\, \lambda^{\,2}\,\left(\,T\,T^{\,\ast} \,+\, L\,L^{\,\ast}\,\right)$, for some \,$\lambda \,>\, 0$.\;Now, for \,$f \,\in\, X_{F}$, 
\[\sum\limits^{\infty}_{i \,=\, 1} \,\left|\,\left <\,f \,,\, f_{\,i} \,+\, g_{\,i} \,|\, a_{\,2} \,,\, \cdots \,,\, a_{\,n}\,\right >\,\right |^{\,2} \,=\, \sum\limits^{\infty}_{i \,=\, 1} \,\left|\,\left <\,f \,,\, T\,e_{\,i} \,+\, L\,e_{\,i} \,|\, a_{\,2} \,,\, \cdots \,,\, a_{\,n}\,\right >\,\right |^{\,2}\]
\[\;=\; \sum\limits^{\infty}_{i \,=\, 1}\,\left |\,\left <\,(\,T \,+\, L\,)^{\,\ast}\,f \,,\, e_{\,i} \,|\, a_{\,2} \,,\, \cdots \,,\, a_{\,n}\,\right >\,\right |^{\,2}\hspace{2.8cm}\]
\[\hspace{1cm} \,=\, \left\|\,\left(\,T \,+\, L\,\right)^{\,\ast}\,f \,,\, a_{\,2} \,,\, \cdots \,,\, a_{\,n}\,\right\|^{\,2}\; \;[\;\text{since}\;\{\,e_{\,i}\,\}\;\text{is F-orthonormal}\;] \]
\[  \,=\, \left <\,\left(\,T \,+\, L\,\right)^{\,\ast}\,f \,,\, \left(\,T \,+\, L\,\right)^{\,\ast}\,f \,|\, a_{\,2} \,,\, \cdots \,,\, a_{\,n}\,\right>\hspace{2.2cm}\]
\[ \,=\, \left <\,T^{\,\ast}\,f \,+\, L^{\,\ast}\,f \,,\, T^{\,\ast}\,f \,+\, L^{\,\ast}\,f \,|\, a_{\,2} \,,\, \cdots \,,\, a_{\,n}\,\right>\hspace{2.1cm}\]
\[\hspace{.2cm} \,=\, \left <\,T^{\,\ast}\,f \,,\, T^{\,\ast}\,f \,|\, a_{\,2} \,,\, \cdots \,,\, a_{\,n}\,\right > \,+\, \left <\,L^{\,\ast}\,f \,,\, T^{\,\ast}\,f \,|\, a_{\,2} \,,\, \cdots \,,\, a_{\,n}\,\right >\]
\[\hspace{2.5cm} \,+\, \left<\,T^{\,\ast}\,f \,,\, L^{\,\ast}\,f \,|\, a_{\,2} \,,\, \cdots \,,\, a_{\,n}\,\right > \,+\, \left <\,L^{\,\ast}\,f \,,\, L^{\,\ast}\, f \,|\, a_{\,2} \,,\, \cdots \,,\, a_{\,n}\,\right >\]
\[\,=\, \left <\,T\,T^{\,\ast}\,f \,,\, f \,|\, a_{\,2} \,,\, \cdots \,,\, a_{\,n}\,\right > \,+\, \left<\,T\,L^{\,\ast}\,f \,,\, f \,|\, a_{\,2} \,,\, \cdots \,,\, a_{\,n}\,\right >\hspace{.3cm}\]
\[\hspace{2.5cm} \,+\, \left <\,L\,T^{\,\ast}\,f \,,\, f \,|\, a_{\,2} \,,\, \cdots \,,\, a_{\,n}\,\right > \,+\, \left<\,L\,L^{\,\ast}\,f \,,\, f \,|\, a_{\,2} \,,\, \cdots \,,\, a_{\,n}\,\right >\]
\[\hspace{2cm} \,\geq\, \left<\,\left(\,T\,T^{\,\ast} \,+\, L\,L^{\,\ast}\,\right)\, f \,,\, f \,|\, a_{\,2} \,,\, \cdots \,,\, a_{\,n}\,\right >\; [\;\text{since}\; T\,L^{\,\ast},\, L\,T^{\,\ast}\;\text{are positive}\;]\]
\[\hspace{2cm} \,\geq\, \dfrac{1}{\lambda^{\,2}}\,\left <\,K\,K^{\,\ast}\,f \,,\, f \,|\, a_{\,2} \,,\, \cdots \,,\, a_{\,n}\,\right>\; [\;\text{since}\;K\,K^{\,\ast} \,\leq\, \lambda^{\,2}\,\left(\,T\,T^{\,\ast} \,+\, L\,L^{\,\ast}\,\right)\;]\]
\[\hspace{1.1cm}\;=\,\dfrac{1}{\lambda^{\,2}}\,\left <\,K^{\ast}\,f \,,\, K^{\,\ast}\,f \,|\, a_{\,2} \,,\, \cdots \,,\, a_{\,n}\,\right > \,=\, \dfrac{1}{\lambda^{\,2}}\,\left \|\,K^{\ast}\,f \,,\, a_{\,2} \,,\, \cdots \,,\, a_{\,n}\,\right \|^{\,2}.\]\; Therefore, for each \,$f \,\in\, X_{F}$,
\begin{equation}\label{eq4}
\dfrac{1}{\lambda^{\,2}}\,\left \|\,K^{\ast}\, f \,,\, a_{\,2} \,,\, \cdots \,,\, a_{\,n}\,\right \|^{\,2} \,\leq\, \sum\limits^{\infty}_{k \,=\, 1} \,\left|\,\left <\,f \,,\, f_{\,i} \,+\, g_{\,i} \,|\, a_{\,2} \,,\, \cdots \,,\, a_{\,n}\,\right>\,\right |^{\,2}. 
\end{equation}
On the other hand, using the Minkowski's inequality, for each \,$f \,\in\, X_{F}$, we have
\[\left(\,\sum\limits^{\infty}_{i \,=\, 1}\,\left|\,\left <\,f \,,\, f_{\,i} \,+\, g_{\,i} \,|\, a_{\,2} \,,\, \cdots \,,\, a_{\,n}\,\right >\,\right|^{\,2}\,\right)^{\dfrac{1}{2}}\]
\[\,\leq\,  \left(\,\sum\limits^{\infty}_{i \,=\, 1}\,\left|\,\left <\,f \,,\, f_{\,i} \,|\, a_{\,2} \,,\, \cdots \,,\, a_{\,n}\,\right >\,\right|^{\,2} \,\right)^{\dfrac{1}{2}} \,+\, \left(\,\sum\limits^{\infty}_{i \,=\, 1}\,\left|\,\left <\,f \,,\, g_{\,i} \,|\, a_{\,2} \,,\, \cdots \,,\, a_{\,n}\,\right >\,\right|^{\,2}\,\right)^{\dfrac{1}{2}}\hspace{1.2cm}\]
\[ \,\leq\, \sqrt{A}\,\|\,f,\, a_{\,2},\, \cdots,\, a_{\,n}\,\| \,+\, \sqrt{B}\, \|\,f,\, a_{\,2},\, \cdots,\, a_{\,n}\,\| \,=\, \left(\,\sqrt{A} \,+\, \sqrt{B}\,\right)\,\|\,f,\, a_{\,2},\, \cdots,\, a_{\,n}\,\|.\]
This implies that
\begin{equation}\label{eq5}
 \sum\limits^{\infty}_{i \,=\, 1} \,\left|\,\left <\,f \,,\, f_{\,i} \,+\, g_{\,i} \,|\, a_{\,2} \,,\, \cdots \,,\, a_{\,n}\,\right >\,\right |^{\,2} \,\leq\, \left(\,\sqrt{A} \,+\, \sqrt{B}\,\right)^{\,2}\,\|\,f \,,\, a_{\,2} \,,\, \cdots \,,\, a_{\,n}\,\|^{\,2}.
\end{equation}
From (\ref{eq4}) \,$\text{and}$\, (\ref{eq5}), \,$\{\,f_{\,i} \,+\, g_{\,i}\,\}^{\infty}_{i \,=\, 1}$\, is a \,$K$-frame associated to \,$\left(\,a_{\,2},\, \cdots,\, a_{\,n}\,\right)$\, for \,$X$.
\end{proof} 

\begin{theorem}
Let \,$\{\,f_{\,i}\,\}^{\infty}_{i \,=\, 1}$\, be K-frame associated to \,$\left(\,a_{\,2},\, \cdots,\, a_{\,n}\,\right)$\, for \,$X$\; and  \,$U \,:\, X_{F} \,\to\, X_{F}$\; be a positive operator.\;Then \,$\{\,f_{\,i} \,+\, U\,f_{\,i}\,\}^{\infty}_{i \,=\, 1}$\, is also a K-frame associated to \,$\left(\,a_{\,2},\, \cdots,\, a_{\,n}\,\right)$\, for \,$X$. 
\end{theorem}

\begin{proof}
Let \,$\{\,f_{\,i}\,\}^{\infty}_{i \,=\, 1}$\; be \,$K$-frame associated to \,$\left(\,a_{\,2},\, \cdots,\, a_{\,n}\,\right)$\, for \,$X$\, with frame operator \,$S_{F}$.\;Then for each \,$f \,\in\, X_{F}$, there exist \,$A,\, B \,>\, 0$\; such that 
\[A \,\|\,K^{\,\ast}\,f \,,\, a_{\,2} \,,\, \cdots \,,\, a_{\,n}\,\|^{\,2} \,\leq\, \sum\limits_{i \,=\, 1}^{\infty}\, \left|\,\left <\,f \,,\, f_{\,i} \,|\, a_{\,2} \,,\, \cdots \,,\, a_{\,n}\,\right >\,\right|^{\,2} \,\leq\, B \,\|\,f \,,\, a_{\,2} \,,\, \cdots \,,\, a_{\,n}\,\|^{\,2}.\]
It is easy to verify that \,$ \left <\,S_{\,F}\,f \,,\,  f \,|\, a_{\,2} \,,\, \cdots \,,\, a_{\,n}\,\right> \,=\, \sum\limits_{i \,=\, 1}^{\infty}\, \left|\,\left <\,f \,,\, f_{\,i} \,|\, a_{\,2} \,,\, \cdots \,,\, a_{\,n}\,\right >\,\right|^{\,2}$.\;Thus,
\[A \;\|\,K^{\,\ast}\,f \,,\, a_{\,2} \,,\, \cdots \,,\, a_{\,n}\,\|^{\,2} \,\leq\, \left <\,S_{\,F}\,f \,,\,  f \,|\, a_{\,2} \,,\, \cdots \,,\, a_{\,n}\,\right> \,\leq\, B \; \|\,f \,,\, a_{\,2} \,,\, \cdots \,,\, a_{\,n}\,\|^{\,2}.\]
This implies that \,$A\,K\, K^{\,\ast} \,\leq\, S_{F} \,\leq\, B\,I$.\;Now, for each \,$f \,\in\, X_{F}$,
\[\sum\limits^{\infty}_{i \,=\, 1}\,\left <\,f \,,\, f_{\,i} \,+\, U\,f_{\,i} \,|\, a_{\,2} \,,\, \cdots \,,\, a_{\,n}\,\right>\, \left(\,f_{\,i} \,+\, U\, f_{\,i}\,\right)\]
\[\hspace{2.3cm}\;=\; \sum\limits^{\infty}_{i \,=\, 1}\,\left <\,f \,,\, (\,I \,+\, U\,)\,f_{\,i} \,|\, a_{\,2} \,,\, \cdots \,,\, a_{\,n}\,\right>\,\left(\,I \,+\, A\,\right)\,f_{\,i}\] 
\[\hspace{2.5cm}\;=\, (\,I \,+\, U\,)\,\sum\limits^{\infty}_{i \,=\, 1}\,\left <\,\left(\,I \,+\, U\,\right)^{\,\ast}\,f \,,\, f_{\,i} \,|\, a_{\,2} \,,\, \cdots \,,\, a_{\,n}\,\right>\,f_{\,i}\]
\[ \,=\, (\,I \,+\, U\,)\,S_{F}\,\left (\,I \,+\, U\,\right)^{\,\ast}\,f.\hspace{1.3cm}\]
This shows that the corresponding frame operator for \,$\{\,f_{\,i} \,+\, U\,f_{\,i}\,\}^{\infty}_{i \,=\, 1}$\, is \,$(\,I \,+\, U\,)\,S_{F}\,\left (\,I \,+\, U\,\right)^{\,\ast}$.\;Since \,$S_{F},\, U$\, are positive operators, \,$(\,I \,+\, U\,)\,S_{F}\,\left (\,I \,+\, U\,\right)^{\,\ast} \,\geq\, S_{F} \,\geq\, A\,K\, K^{\,\ast}$, by Theorem (\ref{th1.01}), \,$\{\,f_{\,i} \,+\, U\,f_{\,i}\,\}^{\infty}_{i \,=\, 1}$\, is a \,$K$-frame associated to \,$\left(\,a_{\,2},\, \cdots,\, a_{\,n}\,\right)$\, for \,$X$.  
\end{proof}

\section{Tight $K$-frame and its properties in $n$-Hilbert space}

\begin{definition}
A sequence \,$\{\,f_{\,i}\,\}_{i \,=\,1}^{\infty}$\, in \,$X$\, is said to be a tight K-frame associated to \,$\left(\,a_{\,2},\, \cdots,\, a_{\,n}\,\right)$\, for \,$X$\, if there exist constants \,$0 \,<\, A \,\leq\, B \,<\, \infty$\, such that 
\begin{equation}\label{eq6}
\sum\limits^{\infty}_{i \,=\, 1}\, \left |\, \left <\,f \,,\, f_{\,i} \,|\, a_{\,2} \,,\, \cdots \,,\, a_{\,n}\,\right >\,\right|^{\,2} \,=\, A\, \left\|\,K^{\,\ast} \,f \,,\, a_{\,2} \,,\, \cdots \,,\, a_{\,n}\,\right\|^{\,2}\; \,\forall\, f \,\in\, X_{F} 
\end{equation}
If \,$A \,=\, 1$, then \,$\{\,f_{\,i}\,\}_{i \,=\,1}^{\infty}$\; is called Parseval K-frame associated to \,$\left(\,a_{\,2} \,,\, \cdots \,,\, a_{\,n}\,\right)$\, for \,$X$.
\end{definition}

\begin{remark}
From (\ref{eq6}), we can write
\[ \sum\limits^{\infty}_{i \,=\, 1}\, \left|\,\left <\,\dfrac{1}{\sqrt{A}}\, f \;,\; f_{\,i} \;|\; a_{\,2} \;,\; \cdots \;,\; a_{\,n}\,\right >\,\right |^{\,2} \,=\,\left\|\,K^{\,\ast} \,f \,,\, a_{\,2} \,,\, \cdots \,,\, a_{\,n}\,\right\|^{\,2}.\]
Therefore, if \,$\{\,f_{\,i}\,\}_{i \,=\,1}^{\infty}$\, is a tight K-frame associated to \,$\left(\,a_{\,2},\, \cdots,\, a_{\,n}\,\right)$\, with bound \,$A$\, then the family \,$\left\{\,\dfrac{1}{\,\sqrt{A}}\,f_{\,i}\,\right\}^{\infty}_{i \,=\, 1}$\, is a Parseval K-frame associated to \,$\left(\,a_{\,2},\, \cdots,\, a_{\,n}\,\right)$\, for \,$X$.
\end{remark}

\begin{theorem}
Let \,$\{\,f_{\,i}\,\}_{i \,=\,1}^{\infty}$\; be a tight frame associated to \,$\left(\,a_{\,2},\, \cdots,\, a_{\,n}\,\right)$\, for \,$X$\; with bound \,$A$\, and \,$K \,\in\, \mathcal{B}\,(\,X_{F}\,)$, then \,$\{\,K\,f_{\,i}\,\}_{i \,=\,1}^{\infty}$\; is a tight K-frame associated to \,$\left(\,a_{\,2},\, \cdots,\, a_{\,n}\,\right)$\, for \,$X$\, with bound \,$A$. 
\end{theorem}

\begin{proof}
Since  \,$\{\,f_{\,i}\,\}_{i \,=\,1}^{\infty}$\; is a tight frame associated to \,$\left(\,a_{\,2},\, \cdots,\, a_{\,n}\,\right)$\, for \,$X$\, with bound \,$A$, for any \,$f \,\in\, X_{F}$, we have
\[\sum\limits^{\infty}_{i \,=\, 1}\, \left |\, \left <\,f \,,\, K\,f_{\,i} \,|\, a_{\,2} \,,\, \cdots \,,\, a_{\,n}\,\right >\,\right |^{\,2} \,=\, \sum\limits^{\infty}_{i \,=\, 1}\, \left|\,\left <\,K^{\,\ast}\,f \,,\, f_{\,i} \,|\, a_{\,2} \,,\, \cdots \,,\, a_{\,n}\,\right >\,\right |^{\,2}\]
\[\hspace{4.3cm}=\, A\,\|\,K^{\,\ast}\,f \,,\, a_{\,2} \,,\, \cdots \,,\, a_{\,n}\|^{\,2}\]and hence \,$\{\,K\,f_{\,i}\,\}_{i \,=\,1}^{\infty}$\; is a tight \,$K$-frame associated to \,$\left(\,a_{\,2},\, \cdots,\, a_{\,n}\,\right)$\, for \,$X$\, with bound \,$A$.     
\end{proof}

\begin{theorem}
Let \,$K,\, T \,\in\, \mathcal{B}\,(\,X_{F}\,)$\, and \,$\{\,f_{\,i}\,\}_{i \,=\,1}^{\infty}$\, be a tight K-frame associated to \,$\left(\,a_{\,2},\, \cdots,\, a_{\,n}\,\right)$\, for \,$X$\, with bound \,$A$.\;Then \,$\{\,T\,f_{\,i}\,\}_{i \,=\,1}^{\infty}$\; is a tight T\,K-frame associated to \,$\left(\,a_{\,2},\, \cdots,\, a_{\,n}\,\right)$\, for \,$X$\; with bound \,$A$.
\end{theorem}

\begin{proof}
Since  \,$\{\,f_{\,i}\,\}_{i \,=\,1}^{\infty}$\, is a tight\,$K$-frame associated to \,$\left(\,a_{\,2},\, \cdots,\, a_{\,n}\,\right)$\, for \,$X$\, with bound \,$A$, for any \,$f \,\in\, X_{F}$, we have
\[\sum\limits^{\infty}_{i \,=\, 1}\, \left|\,\left <\,f \,,\, T\,f_{\,i} \,|\, a_{\,2} \,,\, \cdots \,,\, a_{\,n}\,\right>\,\right |^{\,2} \,=\, \sum\limits^{\infty}_{i \,=\, 1}\, \left |\, \left <\,T^{\,\ast}\,f \,,\, f_{\,i} \,|\, a_{\,2} \,,\, \cdots \,,\, a_{\,n}\,\right>\,\right |^{\,2}\]
\[\hspace{2.5cm}=\, A\,\|\,K^{\,\ast}\,(\,T^{\,\ast}\,f\,),\, a_{\,2},\, \cdots,\, a_{\,n}\|^{\,2} =\; A\,\|\,(\,T\,K\,)^{\,\ast}\,f,\, a_{\,2},\, \cdots,\, a_{\,n}\,\|^{\,2}\] 
and hence \,$\{\,T\,f_{\,i}\,\}_{i \,=\,1}^{\infty}$\, is a tight \,$T\,K$-frame associated to \,$\left(\,a_{\,2},\, \cdots,\, a_{\,n}\,\right)$\, for \,$X$\, with bound \,$A$.
\end{proof}

\begin{theorem}
Let \,$\{\,f_{\,i}\,\}_{i \,=\,1}^{\infty}$\; be a tight K-frame associated to \,$\left(\,a_{\,2},\, \cdots,\, a_{\,n}\,\right)$\, for \,$X$\, with bound \,$A$.\;Then there exists a Bessel sequence \,$\{\,g_{\,i}\,\}_{i \,=\,1}^{\infty}$\, associated to \,$\left(\,a_{\,2},\, \cdots,\, a_{\,n}\,\right)$\, with bound \,$B$\, such that for all \,$f \,\in\, X_{F}$,
\[K\,(\,f\,) \,=\, \sum\limits^{\infty}_{i \,=\, 1}\, \left <\,f \;,\; \,g_{\,i} \;|\; a_{\,2} \,,\, \cdots \,,\, a_{\,n}\,\right >\,f_{\,i}\, \;\text{and}\; \;A\,B \,\geq\, 1.\] 
\end{theorem}

\begin{proof}
Let \,$\{\,f_{\,i}\,\}_{i \,=\,1}^{\infty}$\, be a tight \,$K$-frame associated to \,$\left(\,a_{\,2},\, \cdots,\, a_{\,n}\,\right)$\, for \,$X$\, with bound \,$A$.\;Then by Theorem (2.4)\, of \,(\cite{Janfada}), there exists a Bessel sequence \,$\{\,g_{\,i}\,\}_{i \,=\,1}^{\infty}$\, associated to \,$\left(\,a_{\,2},\, \cdots,\, a_{\,n}\,\right)$\, with bound \,$B$\, such that
\[K\,(\,f\,) \,=\, \sum\limits^{\infty}_{i \,=\, 1}\, \left <\,f \,,\, \,g_{\,i} \,|\, a_{\,2} \,,\, \cdots \,,\, a_{\,n}\,\right >\,f_{\,i}\, \;\;\text{and}\]
\[K^{\,\ast}\,(\,f\,) \,=\, \sum\limits^{\infty}_{i \,=\, 1}\, \left <\,f \,,\, \,f_{\,i} \,|\, a_{\,2} \,,\, \cdots \,,\, a_{\,n}\,\right >\,g_{\,i}\; \;\forall\; f \,\in\, X_{F}.\] Since \,$\{\,f_{\,i}\,\}_{i \,=\,1}^{\infty}$\, is a tight \,$K$-frame associated to \,$\left(\,a_{\,2},\, \cdots,\, a_{\,n}\,\right)$\, for \,$X$, we have 
\[ \sum\limits^{\infty}_{i \,=\, 1}\, \left|\,\left<\,f \,,\, f_{\,i} \,|\, a_{\,2} \,,\, \cdots \,,\, a_{\,n}\,\right>\,\right|^{\,2} \,=\, A\, \|\,K^{\,\ast} \,f \,,\, a_{\,2} \,,\, \cdots \,,\, a_{\,n}\,\|^{\,2}\]
\[=\, A\, \left\|\,\sum\limits^{\infty}_{i \,=\, 1}\,\left<\,f \,,\, f_{\,i} \,|\, a_{\,2} \,,\, \cdots \,,\, a_{\,n}\,\right>\,g_{\,i} \;,\; a_{\,2} \,,\, \cdots \,,\, a_{\,n}\,\right\|^{\,2}\hspace{4.5cm}\]
\[=\, A\,\sup\limits_{\|\,g \,,\, a_{\,2} \,,\, \cdots \,,\, a_{\,n}\,\| \,=\, 1}\,\left\{\,\left\|\;\sum\limits^{\infty}_{i \,=\, 1}\,\left<\,\left <\,f \,,\, \,f_{\,i} \,|\, a_{\,2} \,,\, \cdots \,,\, a_{\,n}\,\right >\,g_{\,i} \;,\; g \,|\, a_{\,2} \,,\, \cdots \,,\, a_{\,n} \,\right>\,\right\|^{\,2}\right\}\hspace{.2cm}\]
\[ =\; A\,\sup\limits_{\|\,g \,,\, a_{\,2} \,,\, \cdots \,,\, a_{\,n}\,\| \,=\, 1}\,\left\{\,\left\|\,\sum\limits^{\infty}_{i \,=\, 1}\,\left <\,f \,,\, \,f_{\,i} \,|\, a_{\,2} \,,\, \cdots \,,\, a_{\,n}\,\right >\,\left<\,g_{\,i} \,,\, g \,|\, a_{\,2} \,,\, \cdots \,,\, a_{\,n} \,\right>\,\right\|^{\,2}\right\}\hspace{.2cm}\]
\[\hspace{.3cm}\leq\, A\,\sup\limits_{\|\,g \,,\, a_{\,2} \,,\, \cdots \,,\, a_{\,n}\,\| \,=\, 1}\,\left\{\,\sum\limits^{\infty}_{i \,=\, 1}\,\left|\,\left <\,f \,,\, \,f_{\,i} \,|\, a_{\,2} \,,\, \cdots \,,\, a_{\,n}\,\right >\,\right|^{\,2}\, \sum\limits^{\infty}_{i \,=\, 1}\,\left|\,\left<\,g_{\,i} \,,\, g \,|\, a_{\,2} \,,\, \cdots \,,\, a_{\,n} \,\right>\,\right|^{\,2}\right\}\]
\[\leq\, A\,\sup\limits_{\|\,g \,,\, a_{\,2} \,,\, \cdots \,,\, a_{\,n}\,\| \,=\, 1}\,\left\{\,\sum\limits^{\infty}_{i \,=\, 1}\,\left|\,\left <\,f \,,\, \,f_{\,i} \,|\, a_{\,2} \,,\, \cdots \,,\, a_{\,n}\,\right >\,\right|^{\,2}\; B\, \|\, g \;,\; a_{\,2} \;,\; \cdots \;,\; a_{\,n}\,\|^{\,2}\,\right\}\]
\[[\;\text{since \,$\{\,g_{\,i}\,\}_{i \,=\,1}^{\infty}$\, is a Bessel sequence associated to \,$\left(\,a_{\,2},\, \cdots,\, a_{\,n}\,\right)$\, with bound \,$B$}\;]\]
\[=\, A\,B\,\sum\limits^{\infty}_{i \,=\, 1}\,\left|\,\left <\,f \;,\; \,f_{\,i} \;|\; a_{\,2} \;,\; \cdots \;,\; a_{\,n}\,\right >\,\right|^{\,2}.\hspace{6.5cm}\]The above calculation shows that \,$A\,B \,\geq\, 1$.
\end{proof}

\begin{theorem}
Let \,$\{\,f_{\,i}\,\}_{i \,=\,1}^{\infty}$\, and \,$\{\,g_{\,i}\,\}_{i \,=\,1}^{\infty}$\; be two Parseval K-frame associated to \,$\left(\,a_{\,2},\, \cdots,\, a_{\,n}\,\right)$\, for \,$X$\; with the corresponding synthesis operators \,$T$\, and \,$L$, respectively.\;If \,$T\,L^{\,\ast} \,=\, \theta$, where \,$\theta$\, is the null operator on \,$X_{F}$\, then \,$\{\,f_{\,i} \,+\, g_{\,i}\,\}_{i \,=\,1}^{\infty}$\, is a tight K-frame associated to \,$\left(\,a_{\,2},\, \cdots,\, a_{\,n}\,\right)$\, with frame bound \,$2$.
\end{theorem}

\begin{proof}
Let \,$\{\,f_{\,i}\,\}_{i \,=\,1}^{\infty}\, \;\text{and}\; \,\{\,g_{\,i}\,\}_{i \,=\,1}^{\infty}$\; be Parseval \,$K$-frames associated to \,$\left(\,a_{\,2},\, \cdots,\, a_{\,n}\,\right)$\, for \,$X$.\;Then by Theorem (\ref{th7}), there exist synthesis operators \,$T$\, and \,$L$\; such that \,$T\,e_{\,i} \,=\, f_{\,i},\; L\,e_{\,i} \,=\, g_{\,i}$\, with \,$\mathcal{R}\,(\,K\,) \,\subset\, \mathcal{R}\,(\,T\,),\; \mathcal{R}\,(\,K\,) \,\subset\, \mathcal{R}\,(\,L\,)$\, respectively, where \,$\{\,e_{\,i}\,\}^{\infty}_{i \,=\, 1}$\, is a \,$F$-orthonormal basis for \,$l^{\,2}\,(\,\mathbb{N}\,)$.\;Now, for \,$f \,\in\, X_{F}$,
\[T^{\,\ast} \,:\, X_{F} \,\to\, l^{\,2}\,(\,\mathbb{N}\,),\; T^{\,\ast}\,(\,f\,) \,=\, \sum\limits^{\infty}_{i \,=\, 1}\,\left <\,f \,,\, f_{\,i} \,|\, a_{\,2} \,,\, \cdots \,,\, a_{\,n}\,\right>\,e_{\,i},\, \;\text{and}\]
\[L^{\,\ast} \,:\, X_{F} \,\to\, l^{\,2}\,(\,\mathbb{N}\,),\; L^{\,\ast}\,(\,f\,) \,=\, \sum\limits^{\infty}_{i \,=\, 1}\,\left <\,f \,,\, g_{\,i} \,|\, a_{\,2} \,,\, \cdots \,,\, a_{\,n}\,\right>\,e_{\,i}.\]
Now from the definition of Parseval \,$K$-frame associated to \,$\left(\,a_{\,2},\, \cdots,\, a_{\,n}\,\right)$,
\begin{equation}\label{eq7} 
\left \|\,K^{\,\ast}f,\, a_{2},\, \cdots,\, a_{n}\,\right \|^{2} = \sum\limits^{\infty}_{i \,=\, 1}\,\left|\,\left <\,f,\, f_{\,i} \,|\, a_{2},\, \cdots,\, a_{n}\,\right >\,\right |^{2}\,=\, \left \|\,T^{\,\ast}f,\, a_{2},\, \cdots,\, a_{n}\,\right\|^{\,2},
\end{equation}

\begin{equation}\label{eq8} 
\left \|\,K^{\,\ast}f,\, a_{2},\, \cdots,\, a_{n}\,\right \|^{2} = \sum\limits^{\infty}_{i \,=\, 1}\,\left|\,\left <\,f,\, f_{\,i} \,|\, a_{2},\, \cdots,\, a_{n}\,\right >\,\right |^{2}\,=\, \left \|\,L^{\,\ast}f,\, a_{2},\, \cdots,\, a_{n}\,\right\|^{\,2}.
\end{equation}
Following the proof of the Theorem (\ref{th8}), it can be shown that for each \,$f \,\in\, X_{F}$, 
\[ \sum\limits^{\infty}_{i \,=\, 1}\, \left|\,\left <\,f \,,\, f_{\,i} \,+\, g_{\,i} \,|\, a_{\,2} \,,\, \cdots \,,\, a_{\,n}\,\right>\,\right |^{\,2} \,=\, \sum\limits^{\infty}_{i \,=\, 1}\,\left |\,\left <\,f \,,\, \left(\,T \,+\, L\,\right)\,e_{\,i} \,|\, a_{\,2} \,,\, \cdots \,,\, a_{\,n}\,\right>\,\right |^{\,2} \]
\[\hspace{4.5cm}\,=\, \left \|\,\left(\,T \,+\, L\,\right)^{\,\ast}\, f \,,\, a_{\,2} \,,\, \cdots \,,\, a_{\,n}\,\right\|^{\,2}\]
\[\;=\, \left <\,T\,T^{\,\ast}\, f \,,\, f \,|\, a_{\,2} \,,\, \cdots \,,\, a_{\,n}\right> \,+\, \left <\,T\, L^{\,\ast}\,f \,,\, f \,|\, a_{\,2} \,,\, \cdots \,,\, a_{\,n}\right> \,+\, \left <\,L\,T^{\,\ast}\, f \,,\, f \,|\, a_{\,2} \,,\, \cdots \,,\, a_{\,n}\right>\]
\[\hspace{.2cm} \,+\, \left <\,L\,L^{\,\ast}\,f \,,\, f \,|\, a_{\,2} \,,\, \cdots \,,\, a_{\,n}\right> \]
\[ \,=\, \left <\,T\, T^{\,\ast}\,f \,,\, f \,|\, a_{\,2} \,,\, \cdots \,,\, a_{\,n}\right> \,+\, \left <\,L\,L^{\,\ast}\,f \,,\, f \,|\, a_{\,2} \,,\, \cdots \,,\, a_{\,n}\right>\; [\;\text{since}\;  \,T\,L^{\ast} \,=\, \theta \,=\, L\,T^{\,\ast}\;]\]
\[ \,=\, \left <\,T^{\,\ast}\,f \,,\, T^{\,\ast}\,f \,|\, a_{\,2} \,,\, \cdots \,,\, a_{\,n}\,\right > \,+\, \left <\,L^{\,\ast}\, f \,,\, L^{\,\ast}\,f \,|\, a_{\,2} \,,\, \cdots \,,\, a_{\,n}\,\right > \hspace{3.8cm}\]
\[\,=\, \left \|\,T^{\,\ast}\,f \,,\, a_{\,2} \,,\, \cdots \,,\, a_{\,n}\,\right \|^{\,2} \,+\, \left\|\,L^{\,\ast}\,f \,,\, a_{\,2} \,,\, \cdots \,,\, a_{\,n}\,\right \|^{\,2}\hspace{7.8cm}\]
\[\,=\, \left \|\,K^{\,\ast}\,f \,,\, a_{\,2} \,,\, \cdots \,,\, a_{\,n}\,\right \|^{\,2} \,+\, \left \|\,K^{\,\ast}\,f \,,\, a_{\,2} \,,\, \cdots \,,\, a_{\,n}\,\right \|^{\,2}\; [\;\text{using (\ref{eq7}) and (\ref{eq8})}\;] \hspace{5cm}\]
\[\,=\, 2\,\left \|\,K^{\,\ast}\,f \,,\, a_{\,2} \,,\, \cdots \,,\, a_{\,n}\,\right \|^{\,2}. \hspace{11cm}\]
Hence, \,$\{\,f_{\,i} \;+\; g_{\,i}\,\}_{i \,=\,1}^{\infty}$\; is a tight\;$K$-frame associated to \,$\left(\,a_{\,2},\, \cdots,\, a_{\,n}\,\right)$\; for \,$X$\; with bound \,$2$.
\end{proof}

\end{document}